\documentclass{amsart}
\usepackage{amsmath}
\usepackage{amssymb}
\usepackage{amsthm}
\usepackage{amsaddr}
\usepackage{mathrsfs}
\usepackage[all,cmtip]{xy}
\usepackage{times}
\usepackage{graphicx}
\usepackage{enumerate}
\numberwithin{equation}{section}

\usepackage{graphicx}
\usepackage{tikz-cd}
\usepackage[all,cmtip]{xy}
\usepackage{mathrsfs}
\usepackage{mathtools}

\def\lhd{\triangleleft}
\def\ra{\rightarrow}
\def\xra{\xrightarrow}
\def\iff{\Longleftrightarrow}

\def\emb{\hookrightarrow}

\newcommand{\sslash}{\mathbin{\! / \mkern-6mu/  \!}}

\def\A{{\mathcal A}}
\def\B{{\mathcal B}}

\def\C{{\mathcal C}}

\def\dom{{\mathbf d}}

\def\F{{\mathcal F}}

\def\H{{\mathscr H}}

\def\L{{\mathfrak L}}
\def\M{{\mathcal M}}

\def\Z{{\mathbb Z}}

\def\ZG{{\mathbb Z}G}

\def\varep{\varepsilon}
\def\ph{\varphi}

\def\st{\operatorname{star}}
\def\cost{\operatorname{costar}}

\def\id{\operatorname{id}}
\def\im{\operatorname{im}}

\def\Der{\operatorname{Der}}

\def\gpd{\sf \textbf{Gpd}}
\def\homogpd{\textbf{OGpd}}

\newcommand{\Mod}[1]{\operatorname{Mod}_{#1}}
\newcommand{\Ext}[1]{\operatorname{Ext}_{#1}}

\def\dom{\mathbf{d}}
\def\ran{\mathbf{r}}

\def\leq{\leqslant}
\def\geq{\geqslant}
\def\<{\langle}
\def\>{\rangle}

\newtheorem{theorem}{Theorem}[section]
\newtheorem{prop}[theorem]{Proposition}
\newtheorem{lemma}[theorem]{Lemma}
\newtheorem{cor}[theorem]{Corollary}

\theoremstyle{definition}
\newtheorem{remark}[theorem]{Remark}
\newtheorem{example}[theorem]{Example}

\newtheorem{definition}[theorem]{Definition}

\title{Cohomology and extensions of ordered groupoids}

\author{B. O. Bainson and N. D. Gilbert}

\address{
School of Mathematical and Computer Sciences\\
and the Maxwell Institute for the Mathematical Sciences,\\
Heriot-Watt University, Edinburgh EH14 4AS, U.K.}
\email{bob30@hw.ac.uk, N.D.Gilbert@hw.ac.uk}

\date{}
\thanks{Some of these results are presented in a different form as part of the first author's PhD thesis \cite{bob}.  
The support of a MACS Global Platform Studentship from Heriot-Watt University is gratefully acknowledged.}

\begin{document}

\begin{abstract} 
We adapt and generalise results of Loganathan on the cohomology of inverse semigroups to the cohomology of ordered groupoids.  We then derive a five-term exact sequence in cohomology from an extension of ordered groupoids, and show that
this sequence leads to a classification of extensions by a second cohomology group.  Our methods use structural ideas in cohomology as far as possible, rather than computation with cocycles.
\end{abstract}


\maketitle

\thispagestyle{empty}

\section*{Introduction} 
Ordered groupoids arose in the work of Ehresmann \cite{Ehres} on the foundations of differential geometry, as an algebraic model of a
pseudogroup of transformations in which the ordered structure models restriction of domain.  Groupoids as models of
symmetry are discussed in \cite{Wei}, and ordered groupoids offer a general algebraic framework for the discussion of partial symmetry.
There are close connections with the theory of inverse semigroups, since the category of inverse semigroups can be viewed 
as a subcategory of the category of ordered groupoids.  These connections are an important theme of the book \cite{MV}
by Lawson, and recent research has both generalised known results from inverse semigroups to ordered groupoids
(see, for example \cite{Gilb, Lw1}), and has 
used ordered groupoids as a more natural setting for the development of structural ideas (see, for example,
\cite{AGM, MO, LwSt, St}).

A cohomology theory for inverse semigroups was introduced by Lausch \cite{Ls}, who used it to classify extensions
by an inverse semigroup $S$ of an $S$-module $\A$, unifying earlier results of Coudron \cite{Cou} and
D'Alarcao \cite{DA}.  Lausch's paper introduces the category $\Mod{S}$ of $S$--modules, establishes that it has enough
injectives, constructs a cohomology functor from $\Mod{S}$ to abelian groups, and constructs projective resolutions of a 
homogeneous module $\Delta \Z$ to compute cohomology.  He then shows directly, using 
a $2$--cocycle or \textit{factor set} associated to an extension,
that extensions of $\A$ by $S$ are classified by the cohomology group $H^2(S^1,\A^0)$ of the inverse monoid $S^1$.

Loganathan \cite{Lg} recognised that Lausch's category $\Mod{S}$ could be obtained as the category of modules of a
left-cancellative category $\L(S)$ derived from $S$ (and that $\L(S)$ was equivalent to an earlier construction of
Leech \cite{Lee}).  Loganathan then developed a number of results about the homology and cohomology of $\L(S)$,
particularly concerning its relationship to the group (co)homology of the maximal group image $\hat{S}$ of $S$,
and the cohomology of the semilattice $E(S)$ of idempotents of $S$.  In a subsequent paper \cite{Lg2} (in the more general 
setting of regular semigroups), Loganathan recovers Lausch's classification of extensions, again using an essentially
computational approach based on  cocycles.  In his PhD thesis, J. Matthews \cite{Mthesis} returned to Lausch's approach, to establish a cohomology theory for arbitrary ordered groupoids and to extend the use of factor sets to classify extensions.  The resulting computational complications are considerable, but are well-handled in \cite{Mthesis}.

Our aim in this paper is to generalise Loganathan's account of the cohomology of inverse semigroups \cite{Lg} to
ordered groupoids, and to give a conceptual approach to the classification of extensions that does not involve
computation with cocycles.  This is carried out for groups, for example, in \cite[section 6.10]{HS}.  We shall adapt the
account given by Gruenberg in \cite{Grue_paper,Grue}.    After a review in section \ref{ord_gpd} of the basic facts about 
ordered groupoids that we need, we turn in sections \ref{modules} and \ref{cohom} to modules over ordered groupoids and 
ordered groupoid cohomology.  These two sections are closely based on  \cite{Lg}, and in them we check that Loganathan's constructions of some adjoint functors for modules over inverse semigroups can be generalised to
ordered groupoids, and that key identifications of cohomology groups with certain $\Ext{}$--groups remain valid.  Apart
from these generalisations and adaptations, we make no claim of any originality in these two sections, although we do proffer
some additional detail omitted in \cite{Lg}.  In section \ref{extensions} we discuss extensions of ordered groupoids
involving an identity-separating map.    Our main result (Theorem \ref{five}) is the construction of a five-term exact sequence in cohomology, which is obtained using an intermediate construction of a \textit{derived module} of a functor
between ordered groupoids, a generalisation of Crowell's derived module of a group homomorphism \cite{Crow}.   The five-term
exact sequence is then the principal component in the classification of extensions by $H^2(Q^I,\A^0)$ in section \ref{classify}.
A crucial step is to show that, up to equivalence, any extension of a $Q$--module $\A$ by an ordered groupoid $Q$ can be obtained as a quotient of a semidirect product ${\mathbb F}(Q) \ltimes \A$, where ${\mathbb F}(Q)$ is a free ordered groupoid.  Here we need to
use the quotient of ordered groupoids introduced in \cite{AGM}, whose construction is recalled in section \ref{ord_gpd}.

Related work on the cohomology of categories can be found in \cite{Webb1, Webb2, Xu}.  In particular, Webb \cite{Webb2}
obtains five-term exact sequences in the (co)homology of categories.  However, since the constructions in this paper
originate from an ordered groupoid $G$ rather than from its associated category $\L(G)$, there are resulting differences at key points, such as in the definition of the analogue of the group ring, of the augmentation ideal, and the notion of extension.

\section{Ordered Groupoids}
\label{ord_gpd}
A groupoid $G$ is a small category in which every morphism is invertible. The set of identities of $G$ is denoted by $G_0$, but we shall sometimes use $E(G)$, following the customary notation for the set of idempotents in an  inverse semigroup. We write $g \in G(e,f)$ when $g$ is a morphism starting at $e$ and ending at $f$.  We regard a groupoid as an algebraic structure comprising its morphisms, and compositions of morphisms as a partially defined binary operation (see \cite{HiBook}, \cite{MV}). The identities are then written as $e=g\dom =gg^{-1}$ and $f=g\ran=g^{-1}g$ respectively. The star  and costar of $G$ at an identity $e$ are defined as the sets $\st_{G}(e)=\{ g\in G | gg^{-1}=e \}$ and $\cost_{G}(e)=\{g\in G | g^{-1}g=e\}$ respectively: their intersection is the local group $G_e$.  A groupoid map $\theta:G\ra H$ is just a functor: groupoids together with groupoid maps then constitute the category $\gpd$ of groupoids.

\begin{definition}
An ordered groupoid is a pair $(G, \leq)$ where $G$ is a groupoid and $\leq$ is a partial order defined on $G$, satisfying the following axioms:
\begin{enumerate}
\item[OG1]  $x \leq y  \Rightarrow x^{-1} \leq y^{-1}$, for all $x, y \in G$.
\item[OG2] Let $x, y, u, v\in G$ such that $x\leq y$ and $u\leq v$. Then $xu\leq yv$ whenever the compositions $xu$ and $yv$ exist.
\item[OG3] Suppose $x\in G$ and $e\in G_0$ such that $e\leq x\mathbf{d}$, then there is a unique element $(e|x)$ called the \textit{restriction} of $x$ to $e$ such that $(e|x)\mathbf{d}= e$ and $(e|x)\leq x$.
\item[OG4] If $x\in G$ and $e \in G_0$ such that $e \leq x \mathbf{r}$, then there exist a unique element $(x|e)$ called the \textit{corestriction} of $x$ to $e$ such that $(x|e) \mathbf{r}= e$ and $(x|e) \leq x$. 
\end{enumerate} 
It is easy to see that OG3 and OG4 are equivalent: if OG3 holds then we may define a corestriction $(x|e)$ by $(x|e) = (e|x^{-1})^{-1}$.
\end{definition}
An ordered functor $\phi: G \ra H$ of ordered groupoids is an order preserving groupoid--map, that is $g\phi \leq h\phi$ if $g\leq h$. Ordered groupoids together with ordered functors constitute the category of ordered groupoids, $\homogpd{}$.  
An ordered functor $\phi : G \ra H$ is \textit{identity-separating} if it is injective when restricted to $G_0$.

Suppose $g,h \in G$ and that the greatest lower bound $\ell$ of $g\textbf{r}$ and $h\textbf{d}$ exist, then we define the {\it pseudoproduct} of $g$ and $h$ by $g\ast h= (g|\ell)(\ell|h)$.  An ordered groupoid is called \textit{inductive} if the pair $(G_0,\leq)$ is a meet semilattice. 
In an inductive groupoid $G$, the pseudoproduct is everywhere defined and $(G,\ast)$ is then an inverse semigroup:
see
\cite[Theorem 4.1.8]{MV}

To any ordered groupoid $G$ we associate a category$\L(G)$ as follows. The objects of $\L(G)$ are precisely the objects of $G$ and morphisms are given by pairs $(e,g) \in G_0 \times G$ where $g\dom \leq e$, with $(e,g)\dom=e$ and $(e,g)\ran=g\ran $. The composition of morphisms is defined by the partial product $(e,g)(f,h)= (e, g\ast h)=(e,(g|h \dom)h)$ whenever $g\ran=f$. It is easy to that $\L(G)$ is left cancellative.  
This construction originates in the work of Loganathan \cite{Lg}, and forms the basis of the treatment in \cite{Lg} of the cohomology of inverse semigroups.

Let $G$ be an ordered groupoid and set $G^{I}= G \cup \{I\}$ where $I$ is an identity and $I\notin G$. Set $e\leq I$ for every object $e$ of $G$. It is easy to see that $G^I$ is an ordered groupoid with maximal ordered subgroupoids $G$ and the singleton groupoid $\{I\}$. The associated category $\L(G^I)$ is defined by the set of objects $E(G) \cup \{I\}$ and morphisms comprising of morphisms of $G$ and order maps $\alpha^I_e:I\ra e$ for all $e\in E(G)$ inherited from the extension of the ordering on $G$. 

We shall need the construction of a quotient of an ordered groupoid by a normal ordered subgroupoid given in 
\cite[section 4]{AGM}.

\begin{definition}  An ordered subgroupoid $N$  of an ordered groupoid $G$ is a \textit{normal ordered subgroupoid} if the following axioms are satisfied: \begin{enumerate}
\item[N01] $N$ has the same set of objects as $G$: that is $N$ is \textit{wide}.
\item[N02] Suppose $n\in N$ and $e\leq n\textbf{d}$ for $e\in G_0$, then the restriction $(n|e)$ of $n$ to $e$ is in $N$.
\item[N03] If $n\in N$ and $k,h\in G$ such that: 
\begin{enumerate}
\item[(a)] $k$ and $h$ have an upper bound $g \in G$, that is $k\leq g$ and $h\leq g$,
\item[(b)] $h^{-1}nk$ exists in $G$,
\end{enumerate} then $h^{-1}nk\in N$.
\end{enumerate} 
\end{definition}

A normal ordered subgroupoid $N$ of an ordered groupois $G$ determines an equivalence relation $\simeq_N$ on $G$, defined by 
\begin{align*}
g \simeq_N h  \iff & \; \text{there exist $a,b,c,d \in N$ such that $agb$ and $chd$ are defined in $G$}, \\ 
& \; \text{with $agb \leq h$ and $chd \leq g$.}
\end{align*}
The relation
\[ [g] \leq [h] \iff \text{there exist $a,b \in N$ such that $agb \leq h$} \]
is then a partial order on the set $G \sslash N$ of $\simeq_N$--classes.  If $g^{-1}g \simeq_N hh^{-1}$ then there exist
$p,q \in N$ such that $pp^{-1} \leq g^{-1}g, p^{-1}p=hh^{-1}, q^{-1}q \leq hh^{-1}$, and $qq^{-1}=g^{-1}g$.  Then
$(g|pp^{-1})ph \simeq_N gq(q^{-1}q|h)$ and the composition of $\simeq_N$ classes given by
$[g][h] = [(g|pp^{-1})ph]$ is well-defined and makes $G \sslash N$ into an ordered groupoid (see \cite[Theorem 4.14]{AGM}).

In the case that $N$ is a union of groups, it is easy to see that the relation $\simeq_N$ simplifies to
\[
g \simeq_N h  \iff  \; \text{there exist $a,b \in N$ such that} \; h= agb \,.
\]
In particular, this will apply when $N$ is a $G$--module,

\section{Modules over ordered groupoids}
\label{modules}
We now review the theory of modules over ordered groupoids, following the approach of Loganathan in \cite{Lg} for inverse semigroups. 

\begin{definition} Let $G$ be an ordered groupoid, with associated category $\L(G)$.
A $G$-module $\A$ is a functor $\A:\L(G)\ra \textbf{Ab}$ associating to each $e\in E(G)$ the abelian group $A_e$ and a group homomorphism $\lhd (e,g): A_e \ra A_{g^{-1}g}$ for every morphism $(e,g)\in \L(G)$. The map $\lhd(e,g)$ is the composite of a group homomorphism $\lhd(e,gg^{-1})$ and a group isomorphism $\lhd(gg^{-1},g)$.  Morphisms of $G$--modules are called $G$--maps. We denote the functor category whose objects are $G$--modules and morphisms are $G$--maps by $\Mod{\L(G)}$. 
\end{definition} 

\begin{example}
Let $A$ be an abelian group.  The \textit{constant} module $\Delta A$ is given by $\A_e= A$ for all $e\in E(G)$, and all morphisms
$\lhd(e,g)$ equal to the identity on $A$.
\end{example}

\begin{example}
The \textit{adjoint} module over an ordered groupoid $G$ is defined as the functor $\Z G :\L(G)\ra \textbf{Ab}$ which associates to every $e\in E(G)$ the free abelian group on $\cost_{G}(e)$. A morphism $g\in G(e,y)$ gives a morphism $(\Z G)_e\ra (\Z G)_y$ given on basis elements $h$ of $(\Z G)_e$ by $h\mapsto hg$. For an order map $(e,f)$ we get a morphism $(\Z G)_e\ra (\Z G)_f$ given on basis elements $g$ in $(\Z G)_e$ by the mapping $g \mapsto (f|g)$ where $(f|g)$ is the unique corestriction of $g$ to $f$. 

\begin{lemma}
\label{adjoint_is_free}
If the ordered groupoid $G$ has a maximum identity $1 \in G_0$ then the adjoint module is a free $G$--module (and so is projective).
\end{lemma}

\begin{proof}
A basis element $h \in (\Z G)_e$ is uniquely expressible as $1 \lhd (1,h)$.
\end{proof}

The \textit{augmentation} map is the $G$--map $\Z G\xrightarrow{\epsilon} \Delta \Z$ defined by $\sum_{g^{-1}g=e} n_gg \mapsto \sum_{g^{-1}g=e} n_g$. Its kernel, denoted by $KG$, is called the \textit{augmentation} module. 
 \end{example}  

\begin{lemma} The set of all elements $g-e$ where $g$ is a non-identity morphism in $\cost_{G}(e)$, forms
a $\Z$--basis of the abelian group $(KG)_e$ \end{lemma}

\subsection{Restriction to $G_0$} 
The set of identities $G_0$ is a trivial subgroupoid of $G$, and also a poset, and as such can be regarded as a category $E(G)$
in which there is a unique morphism $x \ra y$ whenever $x \geq y$ in $G_0$.  It is easy to see that $\L(G_0)$ coincides with $E(G)$.
In \cite{Lg}, Loganathan shows that for an inverse semigroup $S$, the inclusion $E(S) \emb \L(S)$ induces a restriction $\Mod{\L(S)}\ra \Mod{E(S)}$ which admits a left adjoint. This construction remains valid for ordered groupoids, as we now show. Let $\B \in \Mod{E(G)}$ and let $B^{(g)}_e$ be a copy of $B_e$ labelled by $g \in G$.  We define 
\[ (\H \B)_e= \bigoplus_{\substack{g\in G \\ g^{-1}g=e}} B^{(g)}_{gg^{-1}}\,.\]
Suppose $g \in G(e,y)$ and let $B^{h}_{hh^{-1}}$ be a summand of $(\H \B)_e$. Then there is a copy $B^{(hg)}_{hh^{-1}}$ of $B_{hh^{-1}}$ in $(\H \B)_y$ labelled by $hg$, and so right multiplication of the labels by $g$ induces a homomorphism $\rho_g: (\H \B)_e\ra (\H \B)_y$. Suppose that $e,f \in G_0$ and that $(e,f)$ is a morphism in $\L(G)$. Let $x \in G$ with $x^{-1}x=e$: then $B^{(x)}_{xx^{-1}}$ is a summand of $(\H \B)_e$. Let $p=(x|f)\dom$. Then $B^{((x|f))}_p$ is a summand of $(\H \B)_f$ and so $(e,f)$ induces a map $(\H \B)_e \ra (\H \B)_f$. In this way, $\H \B$ is a $G$--module. 

\begin{prop}\label{adj} Let $G$ be an ordered groupoid and let $\mathcal{H}$ be the functor $\Mod{E(G)} \ra \Mod{\L(G)}$ defined above.  Then $\H$ is left adjoint to the restriction  $\Mod{\L(G)} \ra \Mod{E(G)}$.\end{prop} 

\begin{proof}
Given an $E(G)$--module $\B$ and an $\L(G)$--module $\C$, we obtain a natural bijection
$\Mod{E(G)}(\B,\C) \ra \Mod{\L(G)}(\H\B,\C)$ as follows.

Given $\psi \in \Mod{\L(G)}(\H\B,\C)$, we define $\phi \in \Mod{E(G)}(\B,\C)$ by setting $\phi_e: B_e \ra C_e$ to be the restriction
$\psi|B_e^{(e)}$.  This obviously defines an $E(G)$--map $\B \ra \C$.  On the other hand, given $\phi \in \Mod{E(G)}(\B,\C)$ we define $\psi \in \Mod{\L(G)}(\H\B,\C)$, where
\[ \psi_e : \bigoplus_{\substack{g \in G \\ g^{-1}g=e}} B_{gg^{-1}} \ra C_e \,,\]
by setting 
\[ \psi_e|B_{gg^{-1}}^{(g)} = \alpha_g^{-1} \phi_{gg^{-1}} \alpha_g \,.\]
where $\alpha_g$ is the action of $g \in G$ mapping $B_{gg^{-1}} \ra B_{g^{-1}g}$ and $C_{gg^{-1}} \ra C_{g^{-1}g}$.
For $e,f \in E(G)$ with $e \geq f$, let $\beta^e_f: B_e \ra B_f$ and $\gamma^e_f; C_e \ra C_f$ be the maps giving the $E(G)$--module structure on $\B$ and $\C$
and, for $g \in G$ with $g^{-1}g =e$, consider the following diagram:

\bigskip
\centerline{
\xymatrixcolsep{4pc}
\xymatrixrowsep{4pc}
\xymatrix{
B^{(g)}_{gg^{-1}} \ar[d]_{\beta^e_f} & B^{(gg^{-1})}_{gg^{-1}} \ar[l]_{\alpha_g} \ar[d]_{\beta^{gg^{-1}}_{(g|f)\dom}}
\ar[r]^{\phi_{gg^{-1}}} & C_{gg^{-1}} \ar[r]^{\alpha_g} \ar[d]^{\gamma^{gg^{-1}}_{(g|f)\dom}} & C_e \ar[d]^{\gamma^e_f}\\
B^{((g|f))}_{(g|f)\dom} \ar[d]_{\alpha_h} & B^{((g|f)\dom)}_{(g|f)\dom} \ar[l]_{\alpha_{(g|f)}} \ar[r]^{\phi_{(g|f)\dom}} \ar@{=}[d]
& C_{(g|f)\dom} 
\ar[r]^{\alpha_{(g|f)}} \ar@{=}[d] & C_f \ar[d]^{\alpha_h} \\
B^{((g|f))}_{(g|f)\dom} & B^{((g|f)\dom)}_{(g|f)\dom} \ar[l]^{\alpha_{g \ast h}} \ar[r]_{\phi_{(g|f)\dom}} & C_{(g|f)\dom} 
\ar[r]_{\alpha_{g \ast h}} & C_{h^{-1}h}
}
}

\bigskip
Each of the six constituent squares here is commutative for obvious reasons, and so the whole diagram commutes, and shows that 
$\psi$ as defined from $\phi$ above is a $G$--map.  Moreover, it is easy to see that the given constructions that define $\phi$ from
$\psi$ and $\psi$ from $\phi$ are inverse.
\end{proof}

\subsection{Ordered groupoids with an adjoined identity} 
The inclusion $G \ra G^I$ induces a functor $\Mod{\L(G^{I})} \ra \Mod{\L(G)}$ which associates each $G^I$--module with the restricted module over the maximal ordered subgroupoid $G$ of $G^I$.  We will show that this has  a right adjoint . A parallel result is found in \cite{Lg} for inverse monoids.

Given  $\A\in \Mod{\L(G)}$ we define $\A^I$ by 
\[ \A^I_e =\left\{
  \begin{array}{lr}
    \A_e & \text{if}~  e\neq I \\
   \lim^{E(G)} \A &\text{if}~ e=I
  \end{array} \right. \] where $e \in E(G) \cup \{ I \}$.

\begin{lemma}Let $G$ be an ordered groupoid and $\A \in \Mod{\L(G)}$. Then $\A^I$ is a $G^I$--module. \end{lemma}

\begin{proof} A morphism in $\L(G^I)$ but not in $\L(G)$ has the form $(I,g)$ for some $g \in G$, and it is the action of such morphisms that we need to explain.  Since $(I,g)=(I,gg^{-1})(gg^{-1},g)$ then it suffices to consider $(I,e)$ for $e \in G_0$: then the map $\lim^{E(G)} \A \ra A_e$ corresponding to $(I,e) \in \L(G^I)$
is the canonical projection $\pi_e$.  Let $\alpha_g : A_{gg^{-1}} \ra A_{g^{-1}g}$ and $\alpha^e_f: A_e \ra A_f$ (where $e \geq f$) be the maps giving the $G$--module structure on $\A$.  Then consider the following diagram, in which $x = (g|hh^{-1})\dom$:

\bigskip
\centerline{
\xymatrixcolsep{4pc}
\xymatrixrowsep{4pc}
\xymatrix{
\lim^{E(G)} \A \ar[r]^{\pi_{gg^{-1}}} \ar[rd]_{\pi_x}
& A_{gg^{-1}} \ar[r]^{\alpha_g} \ar[d]^{\alpha^{gg^{-1}}_x} & A_{g^{-1}g} \ar[d]^{\alpha^{g^{-1}g}_{hh^{-1}}} \\
& A_x \ar[r]_{\alpha_{(g|hh^{-1})}} & A_{hh^{-1}} \ar[r]_{\alpha_h} & A_{h^{-1}h}
}
}

Since $(gg^{-1},x)(x,(g|hh^{-1}))=(gg^{-1},g)(g^{-1}g,hh^{-1}) \in \L(G)$, the central square commutes, and the left-hand triangle commutes by the definition of $\lim$. It follows that $\A^I$ is a $G$--module,
\end{proof}

\begin{prop}The functor $\Mod{\L(G)}\ra \Mod{\L(G^I)}$ defined by $\A \mapsto \A^I$ is  right adjoint to the restriction $\Mod{\L(G^I)}\ra \Mod{\L(G)}$. \end{prop}

\begin{proof}
Let $\B$ be a $G^I$--module and let $\C$ be a $G$--module.  We exhibit a natural bijection
\[ \Mod{\L(G)}(\B|_{\L(G)},\C) \ra \Mod{\L(G^I)}(\B,\C^I) \,.\]
Given a $G$--map $\phi: \B|_{\L(G)} \ra \C$ we define $\psi : \B \ra \C^I$ by extending $\phi$: so $\psi_e = \phi_e$ for $e \ne I$ and
$\psi_I = \beta^I_e \phi_e$.  Since $\phi$ is a $G$--map the following diagram commutes:

\bigskip
\centerline{
\xymatrixcolsep{4pc}
\xymatrixrowsep{1pc}
\xymatrix{
& B_e \ar[r]_{\phi_e} \ar[dd]^{\beta^e_f} & C_e \ar[dd]^{\gamma^e_f} \\ B_I \ar[ur]^{\beta^I_e} \ar[dr]_{\beta^I_f}
& & \\ & B_f \ar[r]_{\phi_f} & C_f
}}
and so there exists an induced map $\psi_I : \B_I \ra \lim^{E(G)} \C$.  This completes the definition of $\psi$ as
a $G^I$--map since the diagram

\bigskip
\centerline{
\xymatrixcolsep{4pc}
\xymatrixrowsep{2pc}
\xymatrix{
B_I \ar[r]_{\psi_I} \ar[d]_{\beta^I_e} & \lim^E \C \ar[d]^{\pi_{gg^{-1}}} \\ B_{gg^{-1}} \ar[d]_{\alpha_g} \ar[r]_{\phi_{gg^{-1}}} & C_{gg^{-1}} \ar[d]_{\alpha_g} \\ B_{g^{-1}g} \ar[r]_{\phi_{g^{-1}g}} & C_{g^{-1}g}
}}
then also commutes.  The correspondence $\phi \mapsto \psi$ is then easily seen to be a bijection, as required.
\end{proof}

Let $\A^0$ be the $G^I$--module obtained from the $G$--module $\A$ by associating to $I$ the trivial group $\{ 0 \}$. It is clear that $\A^0$ admits an $\L(G^I)$--action.  The functor $\Mod{\L(G)} \ra \Mod{\L(G^I)}$ defined by $\A\mapsto \A^0$ is then left adjoint to the restriction $\Mod{\L(G^I)} \ra \Mod{\L(G)}$.

\section{Cohomology of ordered groupoids}
\label{cohom}
We now apply the concept of cohomology of small categories as discussed in \cite{GZ} and \cite{BD} to ordered groupoids, still following Loganathan's approach in \cite{Lg}. Suppose $\A \in \Mod{\L(G)}$. We have $ \Ext{\L(G)}^n(\Delta \Z, \A)= H^n(\Mod{\L(G)}(P, \A))$ where $P$ is a projective resolution of $\Delta \Z$. The \textit{cohomology} of $G$ with coefficients in $\A$ is defined by  
\begin{equation} \label{cohomdef} H^n(G,\A)= \Ext{\L(G)}^n(\Delta \Z, \A) \; .\end{equation}
Properties of $\Ext{}$ as the derived functor of $\lim$ then give the following result.   

\begin{prop}Let $G$ be an ordered groupoid and $\A \in \Mod{\L(G)}$. Then
\begin{enumerate}
\item  $H^0(G,\A)= \lim^{\L(G)} \A$.
\item  $H^n(G,\A)=0$ for $n>0$ and $\A$ injective. 
\end{enumerate}
\end{prop}

\begin{theorem}\label{semil} Suppose $G$ is an ordered groupoid and let $\A$ be a $G$-module. Then there are natural isomorphisms $H^n(E(G),\A) \cong \Ext{\L(G)}^n(\Z G,\A)$. \end{theorem} 

\begin{proof} The proof is adapted fom \cite{Lg}. The category $\Mod{E(G)}$ has enough projectives and so every $E(G)$--module admits a projective resolution. Choose an $E(G)$--projective resolution $P$ of $\Delta \Z$. The functor  $\H$ from Proposition \ref{adj} is left adjoint to restriction  and so preserves projectives: hence $\H P$ is a projective resolution of $\H \Delta \Z = \ZG$. Then 
\begin{align*}
H^n(E(G),\A) = \Ext{E(G)}^n(\Delta \Z, \A) &= H^n \left( \Mod{E(G)}(P, \A) \right) \\
&= H^n \left( \Mod{\L(G)}(\H P, \A) \right) \\
&= \Ext{\L(G)}^n(\H \Delta \Z, \A) =  \Ext{\L(G)}^n(\Z G, \A) \,. 
\end{align*}
\end{proof}

Suppose $G^I$ is an ordered groupoid obtained from $G$ by adjoining the identity $I$ and consider the augmentation modules $KG$ and $KG^I$.

\begin{lemma}\label{ext}
For $n \geq 0$, there are canonical isomorphisms,  
\[ \Ext{\L(G^I)}^n(KG^I,\A^0)\cong \Ext{\L(G^I)}^n((KG)^0,\A^0)\cong \Ext{\L(G)}^n(KG ,\A) \,.\] 
\end{lemma}

\begin{proof} The augmentation module $KG^I$ satisfies $(KG^I)_e = (KG)_e$ for $e \ne I$, and
\[ (KG^I)_I = \ker((\Z G^I)_I \ra \Z) = \ker( \Z \ra \Z ) = \{ 0 \} \,.\]
Therefore $KG^I= (KG)^0$, and the result follows.
\end{proof}   

\begin{theorem}\label{identitycohomo} Let  $\A$ be a $G$--module. Then, for $n>0$, we have 
$$H^n (G^I,\A^0) \cong \Ext{\L(G)}^{n-1} (KG,\A).$$ \end{theorem} 

\begin{proof}  
Using \eqref{cohomdef} and Lemma \ref{ext}, the long exact Ext--sequence obtained from the 
short exact sequence $0 \ra KG^I \ra \ZG^I \ra \Delta \Z \ra 0$ of $G^I$--modules is
 \begin{align}\label{lem3} 
\cdots  \ra H^{n}(G^I,\A^0) &\ra \Ext{\L(G^I)}^n(\ZG^I,\A^0)  \ra \Ext{\L(G^I)}^n(KG,\A) \nonumber \\ 
& \ra H^{n+1}(G^I,\A^0) \ra \Ext{\L(G^I)}^{n+1}(\ZG^I,\A^0) \ra \cdots 
\end{align} 
and  for $n=0$ we  have
\[ H^0 (E(G^I),\A^0) = \Ext{\L(G^I)}^0 (\Z G^I,\A^0)=\Mod{\L(G^I)}(\Z G^I,\A^0) \,.\]
Consider a  basis element $g \in (ZG^I)_e$, with $g^{-1}g=e$: then $g = I \lhd (I,g)$, and so for any $G^I$--map
$\psi: \ZG^I \ra \A^0$, we have 
\[ g \psi_e = (I \lhd (I,g))\psi_e = (I \psi_I) \lhd (I,g) = 0 \]
since $(\A^0)_I= \{0\}$.  Hence $\psi$ is trivial, and $\Mod{\L(G^I)}(\Z G^I,\A^0) = \{ 0 \}$. From this, and the fact that $\Z G^I$ is projective by Lemma \ref{adjoint_is_free}, it follows that
$\Ext{\L(G^I)}^n(\Z G^I,\A^0) = 0$ for all $n \geq 0$.  Therefore, the long exact sequence
\eqref{lem3} splits into the isomorphisms claimed in the Theorem.
\end{proof}

\section{Extensions of ordered groupoids}
\label{extensions}
Our aim, in the remainder of the paper, is to show how extensions of ordered groupoids, with quotient $Q$ and kernel a $Q$--module $\A$, are classifed by the second cohomology group $H^2(Q^I,\A^0)$.   This extension theory has appeared in \cite{Mthesis} as a generalisation of that developed by Lausch in \cite{Ls} for inverse semigroups.  We shall give a simpler, more conceptual account which  is based on the approach due to Gruenberg in \cite{Grue} for group extensions.  The main ingredient is a five-term exact sequence in cohomology, associated to an extension, and we shall now proceed to derive this.  The corresponding five-term exact sequence  in group cohomology is well-known and can be found, for example, in \cite[section VI.8]{HS}.

Working in the category of inverse semigroups, Lausch \cite{Ls} defined an extension of a semilattice of abelian groups $\A$ by an inverse semigroup $S$ to be an inverse semigroup $U$ containing $\A$, with an idempotent-separating homorphism 
$\psi : U \ra S$ such that $\A = \{ u \in U : u\psi \in E(S) \}$.  This definition was extended, in the natural way, to ordered 
groupoids by Matthews \cite{Mthesis}.  We shall need a more general version, omitting the requirement that we have a 
semilattice of \textit{abelian} groups.

\begin{definition}
Let $Q$ be an ordered groupoid, and let $N$ be an ordered groupoid that is a disjoint union of groups $N_x, x \in N_0$.  An \textit{extension} of $N$ by $Q$ is an
ordered groupoid $G$ together with an ordered embedding $N \xrightarrow{\iota} G$ and an  identity-separating surjective ordered functor $\phi : G \ra Q$ such that
$N$ is the kernel of $\phi$: that is, $N = \{ g \in G : g\phi \in Q_0 \}$.
\end{definition}

We note some simple consequences of this definition (as in \cite[Proposition10.1]{Mthesis}).

\begin{lemma}
\label{ext_props}
Let $N \xrightarrow{\iota} G \xrightarrow{\phi} Q$ be an extension of $N$ by $Q$.  Then
\begin{enumerate}[(a)]
\item $N$ is a normal ordered subgroupoid of $G$,
\item there is an isomorphism $G \sslash N \ra Q$ of ordered groupoids, mapping $[g]_N \mapsto g \phi$.
\end{enumerate}
\end{lemma}

Two extensions $G$ and $H$ of $N$ by $Q$ are \textit{equivalent} if there exists an ordered functor $\mu$ making
the following diagram commute:
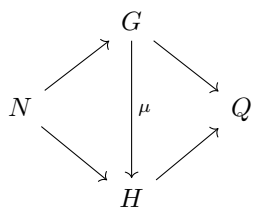
\begin{center} 
\begin{tikzcd} & G\ar{dd}{\mu} \ar{dr} & \\ N \ar{ru} \ar{rd}& & Q \\ & H \ar{ru} & \end{tikzcd}
\end{center} 

Given a $Q$--module $\A$ we can define the {\em semidirect product}
or {\em action groupoid} $Q \ltimes A$ (as defined more generally by Brown \cite{Br2} in the unordered case, and by Steinberg 
\cite{St} for ordered groupoids) which will be an extension of $\A$ by $Q$.  
The set of arrows of $Q \ltimes A$ is $\{ (q,a) : a \in A_{q \ran} \}$, the set of objects is $Q_0$, and the domain and range maps are defined by
$(q,a) \dom = q \dom \; \text{and} \; (q,a)\ran = q\ran$.
The composition is
\[ (g,a)(h,b) = (gh,(a \lhd h)+b) \]
defined when $g \ran = h \dom$, and the ordering is componentwise.  

\subsection{Five-term exact sequences}
The preliminary step in constructing the five-term exact sequence is to assign a short exact sequence of modules to an extension of ordered groupoids. This is a generalisation of the ideas of Crowell in \cite{Crow} for groups. Our approach is adapted from the account in \cite{BHS} for groupoids that are trivially ordered.

The following definitions are taken from \cite{Nouf} and generalise definitions given for inverse semigroups in \cite{Gi4}.
 
\begin{definition} Suppose that $G \xrightarrow{\theta} Q$ is an ordered morphism of ordered groupoids and let $\B$ be a $Q$--module. Then an order-preserving function $f:G\ra \B$ is called an ordered $\theta$--\textit{derivation} if: 
\begin{enumerate}
\item $gf \in \B_{(g^{-1}g)\theta}$,
\item $(gh)f= (gf \lhd h\theta) + hf$, whenever $gh$ is defined in $G$.
\end{enumerate} 
The set $\Der_{\theta}(G, \B)$ of all $\theta$--derivations $G \ra \B$ is an abelian group under pointwise addition.  If
$\theta =  \id_G$ we just write $\Der(G,\B)$ for this group.
\end{definition}

\begin{definition}Suppose that $G\xrightarrow{\theta} Q$ is an ordered morphism. Then its \textit{derived} module is 
a $Q$--module $D_{\theta}$ together with a $\theta$--derivation $\delta$ such that, for each $\theta$--derivation 
$f: G \ra \A$ to a $Q$--module $\A$, there exists a unique $Q$--map $\hat{f} : D_{\theta} \ra \A$ with
$f =\delta \hat{f}$.
 \end{definition} 

\begin{prop}{\cite[Proposition 5.4.1]{Nouf}} 
\label{constr_of_D}
The derived module $D_{\theta}$ together with the derivation $G\xrightarrow{\delta} D_{\theta}$, of an ordered morphism
 $G\xrightarrow{\theta} Q$ exist,  and $D_{\theta}$ is unique up to isomorphism. 
\end{prop}

\begin{proof}
We sketch the construction.  For $e \in H_0$ we define
\[ X_e = \{ (g,q) : g \in G, q \in Q, (g^{-1}g)\theta \geq qq^{-1}, q^{-1}q=e \} \]
and let $F^{ab}(X_e)$ be the free abelian group on $X_e$.  Then we have a $Q$--module $\F^{ab}$
with $(\F^{ab})_e = F^{ab}(X_e)$ and $\L(H)$--action given by $(g,q) \lhd (e,k) = (g,q \ast k)$.  
Note that we need the pseudoproduct here , with $q \ast k = (q|kk^{-1})k$ in this case. We let
$K_e$ be the subgroup of $F^{ab}(X_e)$ generated by all elements
$(g \ast h,q) - (g, h \theta \ast q) - (h,q)$,
where $g^{-1}g \geq hh^{-1}$ in $G_0$, $(h^{-1}h)\theta \geq qq^{-1}$ and $q^{-1}q=e$.  We set
$(D_{\theta})_e = F^{ab}(X_e)/K_e$, and write $\< g,q \>$ for the image in $(D_{\theta})_e$ of $(g,q) \in X_e$.
Now
\[ [(g \ast h,q) - (g, h \theta \ast q) - (h,q)] \lhd (e,k) = (g \ast h,q \ast k) - (g, h \theta \ast q \ast k) - (h,q \ast k) \in K_{k^{-1}k} \]
and so $D_{\theta}$ is a $Q$--module.  The universal derivation $\delta : G \ra D_{\theta}$ is given by
$g \delta = \< g, (g^{-1}g)\theta \>$.
\end{proof}

\begin{example}
\label{der_of_id} 
\leavevmode
\begin{enumerate}[(a)]
\item (\cite[Proposition 5.4.2]{Nouf})
For any ordered groupoid $G$,  the augmentation ideal $KG$ is the derived module of the identity 
map $G\ra G$. 
\item If $N$ is a union of groups, $N = \bigsqcup_{e \in N_0} N_e$, then we have a morphism $\varep: N \ra N_0$ that maps $n \in N_e$ to $e \in N_0$.  The derived
module of $\varep$ is then the abelianisation $N^{ab} = \bigsqcup_{e \in N_0} N_e^{ab}$.
\end{enumerate}
\end{example} 

Let $\mathcal{OG}^2$ be the category of ordered functors between ordered groupoids: its objects 
are ordered functors and its morphisms are commutative squares. The construction of the derived 
module then defines a functor $D$ from $\mathcal{OG}^2$ to the category $\Mod{}$ of modules over ordered groupoids defined by 
\[ D(G\xrightarrow{\theta} Q)= D_{\theta} \in \Mod{Q}\,. \] 
The well-known correspondence between derivations and maps to semidirect products for groups can be phrased for ordered groupoids as follows.

\begin{prop}{\rm \cite[Proposition 5.5.1]{Nouf}} \label{chi} The functor $D$ has a right adjoint 
$\mathcal{X}:\Mod{}\ra \mathcal{OG}^2$ defined, for $\M \in \Mod{G}$, by $\M \mapsto (G\ltimes \M \xrightarrow{p} G)$ 
 \end{prop}

\begin{lemma}
\label{KQ}
Let $N \xrightarrow{\iota} G \xrightarrow{\ph} Q$ be an extension of ordered groupoids.  Then there is a short exact sequence
of $Q$--modules $0 \ra N^{ab} \xra{\bar{\iota}} D_{\ph} \xra{\bar{\ph}} KQ \ra 0$ in which
$\bar{\iota} : \bar{n} \mapsto \< n, (n^{-1}n) \iota \ph \>$ and 
$\bar{\ph} : \< g,q \> \mapsto g \ph \ast q-q$.
\end{lemma}

\begin{proof} 
The exact sequence
\[
\xymatrix{N \ar[r] \ar[d]_{\iota \ph}  & G \ar[r]^{\ph} \ar[d]^{\ph} & Q  \ar[d]^{=} \\ Q_0 \ar[r] & Q \ar[r]_{=} & Q}
\]
in $\mathcal{OG}^2$ can be rewritten as the pushout
\[ \xymatrix{ N \ar[rr]^{\iota} \ar[d]_{\iota \ph} \ar[ddrr] && G \ar[d]_{\ph} \ar[ddrr] && \\  Q_0 \ar[rr]|\hole \ar[ddrr]  && Q \ar[ddrr]|\hole && \\
& & Q_0 \ar[rr] \ar[d]^{=}&& Q \ar[d]^{=} \\ && Q_0 \ar[rr] && Q  } \] 
By Proposition \ref{chi} the functor $D$ is a left adjoint so preserves colimits.  Applying $D$ to the pushout above,
we deduce that the commutative square
\[
\xymatrix{N^{ab} \ar[r]^{\bar{\iota}} \ar[d] & D_{\varphi} \ar[d]^{\bar{\ph}} \\ 0 \ar[r] & KQ } 
\]
is a pushout in $\Mod{Q}$, and so the sequence 
\[N^{ab} \xra{\bar{\iota}} D_{\varphi} \xra{\bar{\ph}} KQ \ra 0 \]
is exact.  To complete the proof, we construct an abelian group homomorphism $\kappa: D_{\ph} \ra N^{ab}$
such that $\bar{\iota} \kappa = \id$.  We use the isomorphism $Q \cong G \sslash N$ from Lemma \ref{ext_props},
and let $\alpha : N \ra N^{ab}$ be the abelianisation.
We choose a transversal $\tau : Q \ra G$ to the equivalence relation
\[
g \simeq_N h  \iff  \; \text{there exist $a,b \in N$ such that} \; h= agb \,.
\]
defining $G \sslash N$.  Since $\ph$ is identity separating, we can identify $Q_0$ and $G_0$, and take $\tau$ to be the identity 
on $Q_0$.  Then for $(g,q) \in X_e$ (as in the proof of Proposition \ref{constr_of_D}), we set
\[ (g,q) \kappa = [((g \ph \ast q)\tau)^{-1} \ast g \ast q \tau] \alpha \in N^{ab} \,. \]
Then if $g,h \in G$ with $g^{-1}g \geq hh^{-1}$, and $q \in Q$ with $(h^{-1}h)\ph \geq qq^{-1}$ and $q^{-1}q=e$ we
apply $\kappa$ to the terms in a defining relation fo $D_{\ph}$ to find:
\begin{align}
\label{defrel1} (g \ast h,q) & \overset{\kappa}{\mapsto} [((g \ph \ast h \ph \ast q)\tau)^{-1} \ast g \ast h \ast q \tau] \alpha \,, \\
\label{defrel2} (g, h\ph \ast q) &\overset{\kappa}{\mapsto} [((g \ph \ast h \ph \ast q)\tau)^{-1} \ast g \ast(h \ph \ast q) \tau] \alpha \,,\\
\label{defrel3} (h,q) & \overset{\kappa}{\mapsto} [((h\ph \ast q)\tau)^{-1} \ast h \ast q \tau] \alpha \,.
\end{align}
Comparing \eqref{defrel1} with \eqref{defrel2} and \eqref{defrel3} we see that $\kappa$ induces an abelian group
homomorphism $\kappa : D_{\ph} \ra N^{ab}$.  Moreover, for $n \in N$, since $n \ph = (nn^{-1})\ph$ we have
\begin{align*}
(n \alpha) \bar{\iota} \kappa &= \< n,(n^{-1}n)\ph \> \kappa = [(n \ph \tau)^{-1} \ast n \ast (n^{-1}n)\ph \tau] \alpha \\
&= [((nn^{-1}) \ph \tau)^{-1}  \ast n \ast (n^{-1}n)\ph \tau] \alpha \\
&= [ nn^{-1}nn^{-1}n]\alpha = n \alpha \,.
\end{align*}
Hence $\bar{\iota}$ is injective,and this completes the proof of the Lemma.
 \end{proof}

\begin{theorem}\label{five} 
Consider an extension 
\[ \mathscr{E} : N \xrightarrow{\mu} G \xrightarrow{\varphi} Q \]
of ordered groupoids, in which $N$ is a union of groups, $\mu$ embeds $N$ as a normal ordered subgroupoid of $G$,
and $Q$ is isomorphic to the quotient $G \sslash N$.
Suppose that $\A$ is a $Q$-module. Then the five-term sequence 
\[ 0\ra \Der(Q,\A)\ra \Der_{\varphi}(G,\A) \ra \Mod{Q}(N^{ab},\A) \ra H^2(Q^I,\A^0) \ra H^2(G^I,\A^0)\] 
is exact. 
\end{theorem} 

\begin{proof} By Lemma \ref{KQ} we obtain the short exact sequence 
 \[ 0 \ra  N^{ab} \ra D_{\varphi} \ra KQ \ra 0 \] 
of $Q$--modules, to which we apply the apply the  functor, $\Ext{Q}(-,\A)$.  We obtain the exact sequence
\[
 0\ra \Mod{Q}(KQ,\A) \ra \Mod{Q}(D_{\varphi},\A) \ra \Mod{Q}(N^{ab},\A) \ra \Ext{Q}^1(KQ,\A) \ra \Ext{Q}^1 (D_{\varphi},\A) 
\] in low dimensions. We make the following identifications to arrive at the desired result.

By Example \ref{der_of_id}(a), the augmentation module $KQ$ is the derived module of the identity map on $Q$, and
its universal property then gives an isomorphism $\Mod{Q}(KQ,\A) \cong \Der_{}(Q,\A)$. 
Similarly $\Mod{Q} (D_{\varphi}, \A) \cong \Der_{\varphi}(G,\A)$.
By Theorem \ref{identitycohomo} we have an isomorphism $\Ext{Q}^1(KQ,\A)  \cong H^2(Q^I,\A^0)$.
The commutative square
\[
\xymatrix{G \ar[r]^{=} \ar[d]_{=} & G\ar[d]^{\varphi} \\ G \ar[r]_{\varphi} & Q}
\]
induces a map $\beta : KG \ra D_{\phi}$, and so also a morphism $\Ext{Q}^i(D_{\varphi},\A) \ra \Ext{G}^i(KG,\A)$. We show that
this map is injective for $i=1$.

We proceed by embedding  $\A$ into some injective $Q$-module $\mathcal{I}$ and set $\mathcal{C}$ to be the quotient
 module. Applying $\Ext{G}^*(KG,-)$ and $\Ext{Q}^*(D_{\varphi},-)$ to the exact sequence $\A \ra \mathcal{I} \ra
 \mathcal{C}$ gives the commutative diagram 
\[
\begin{tikzcd}[column sep=0.65em] \Mod{G}(KG,\A) \rar & \Mod{G}(KG,\mathcal{I})\rar  & \Mod{G}(KG,\mathcal{C}) \rar &\Ext{G}^1(KG,\A)\rar &\Ext{G}^1(KG,\mathcal{I}) \\ \Mod{Q}(D_{\varphi},\A)\ar{u} \rar &\Mod{Q}(D_{\varphi},\mathcal{I}) \rar \ar{u} &\Mod{Q}(D_{\varphi},\mathcal{C})\rar \ar{u} & \Ext{Q}^1(D_{\varphi},\A)\rar \ar{u}{\tau} & 0 \ar{u} \end{tikzcd} 
\]
where the vertical maps are all induced by $\beta$, and we get $0$ at the lower right since ${\mathcal I}$ is injective. 
Now any $Q$--module $\B$ is a $G$--module via $\varphi$, 
and a $Q$--map $D_{\varphi} \ra \B$ exactly corresponds to a $\varphi$--derivation $G \ra \B$, which is a derivation
$G \ra \B$ considering $\B$ as a $G$--module, and so corresponds to a $G$--map $KG \ra \B$.  Hence the first three vertical 
maps in the above diagram are equalities, and a simple diagram chase using exactness of the top line shows that 
$\Ext{Q}^1(D_{\varphi},\A) \xrightarrow{\tau} \Ext{G}^1(KG,\A)$ is injective. Making the identification 
$H^2(G^I,\A^0) \cong \Ext{G}^1(KG,\A)$ given by  Theorem \ref{identitycohomo}, we obtain an embedding
 $\Ext{Q}^1(D_{\varphi}, \A) \emb H^2(G^I,\A^0)$ and the sequence given in the Theorem now follows.
\end{proof}

\section{Classification of extensions}
\label{classify}
 Let $\mathbb{F}(Q)$ be the free ordered groupoid on the underlying graph of the ordered groupoid $Q$. Denote the element of
 $\mathbb{F}(Q)$ that corresponds to an arrow $q\in Q$ by $\lfloor q \rfloor$, and define $\mathbb{F}(Q) \xrightarrow{\pi} Q$ by
 $\lfloor q\rfloor \mapsto q$. We denote its kernel by $\mathbb{N}(Q)$. If $w=\lfloor q_1 \rfloor \lfloor q_2 \rfloor \cdots \lfloor q_m
 \rfloor \in \mathbb{N}(Q)$ then $q_1\cdots q_m \in Q_0$ and so $q_1\textbf{d}=q_m\textbf{r}$ and so
 $w\textbf{d}=w\textbf{r}$. Therefore $\mathbb{N}(Q)$ is a union of groups. 

If $\A$ is a $Q$--module (which we shall write additively) then it is also an $\mathbb{F}(Q)$--module via $\pi$, and so we can
 construct the semidirect product $S=\mathbb{F}(Q) \ltimes \A$. Then $T=\mathbb{N}(Q)\ltimes \A$ is a normal ordered
 subgroupoid of $S$, and since $\mathbb{N}(Q)$ acts trivially, $T$ is just the pullback $T=\{(w,a)\in \mathbb{N}(Q)\times \A :
 w\textbf{r}=a\textbf{d} \}$ with componentwise composition. 

Now $\mathbb{F}(Q)$ acts by conjugation on $\mathbb{N}(Q)$ and, as above, on $\A$  via $\pi$. Hence we can consider the 
set of ordered functors  $\homogpd_{\mathbb{F}(Q)} (\mathbb{N}(Q), \A)$ that respects the $\mathbb{F}(Q)$--actions. This is
 an abelian group under pointwise addition in $\A$.  Given $\phi \in \homogpd_{\mathbb{F}(Q)} (\mathbb{N}(Q), \A)$, we define
 $ M_{\phi}=\{ (w,w\phi): w\in \mathbb{N}(Q)\} \subseteq \mathbb{N}\ltimes \A$. 

\begin{lemma}\label{exten} $M_{\phi}$ is a normal ordered subgroupoid of $S$ and there is an extension 
\[ \mathscr{E}_{\phi}:\A \xrightarrow{\iota} S\sslash M_{\phi} \xrightarrow{\pi_{\phi}} Q \] 
of ordered groupoids. 
\end{lemma} 

\begin{proof} To show that $M_{\phi}$ is a normal ordered subgroupoid of $S$, we verify that $M_{\phi}$ satisfies the axioms
 NO1--NO3. The first axiom NO1 is clear.

Suppose $w\in \mathbb{N}(Q)$ and  $e\in S_0$ such that $e\leq w\textbf{d}$. Then we define the restriction of $(w,w\phi)$ to
 $e$ by $(e|(w,w\phi))=((e|w),(e|w\phi))$. However $(e|w)\phi= (e \phi | w\phi)=(e|w\phi)$ since 
$((e|w)\phi)\textbf{d}=e=e\phi$ and $(e|w)\phi \leq w\phi$. Therefore 
\[ (e|(w,w\phi))= ((e|w),(e|w)\phi) \in M_{\phi} \] 
and NO2 holds

Let $(w,w\phi)\in M_{\phi}$ and $(h,a), (k,b)\in S$ such that $(h,a)$ and $(k,b)$ have an upper bound $(g, c)\in S$ and  let
 $(h,a)^{-1}(w,w\phi)(k,b)$ be defined in  $S$. The subgroupoid $M_{\phi}$ is a disjoint union of groups, and so $(h,a)^{-1
}(w,w\phi)(k,b)$ being defined implies that$(h,a)\textbf{d}=(k,b)\textbf{d}$: since $(h,a)$ and $(k,b)$ are restrictions of $(g,c)$
 in $S$ then $(h,a)=(k,b)$. Thus  
\begin{eqnarray*}
(h,a)^{-1}(w,w\phi)(k,b)&=& (h,a)^{-1}(w,w\phi)(h,a) \\
&=& (h^{-1},-a \lhd h^{-1})(w,w\phi)(h,a)\\
&=& (h^{-1},-a \lhd h^{-1})(wh,(w\phi \lhd h)+a)\\
&=&(h^{-1}wh,(-a \lhd h^{-1}wh)+(w\phi \lhd h)+a).
\end{eqnarray*} 
But $-a\lhd h^{-1}wh=-a$ since $h^{-1}wh \in \mathbb{N}(Q)$ acts trivially and so the second component is 
$w\phi \lhd h$, which is equal to $(h^{-1}wh)\phi$ since $\phi \in \homogpd_{\mathbb{F}(Q)} (\mathbb{N}(Q), \A)$.  Therefore 
 \[ (h,a)^{-1}(w,w\phi)(k,b)= (h^{-1}wh, (h^{-1}wh)\phi)\in M_{\phi} \] and NO3 holds. Therefore $M_{\phi}$  is a normal
 ordered subgroupoid of $S$.  

We now show that $\mathscr{E}_{\phi}$ is an extension of ordered groupoids. Since $M_{\phi}$ is a normal ordered
 subgroupoid  of $S$, the quotient ordered groupoid $S\sslash M_{\phi}$ exists. Define the map $S\sslash M_{\phi}
 \xrightarrow{\pi_{\phi}} Q$ by $[(x ,a)]\mapsto x \pi$.  We show that $\pi_{\phi}$ is well defined. Let $(w,a)\in S$ and suppose
 that $u,v \in \mathbb{N}(Q)$ and that $(u,u\phi)(w,a)(v,v\phi)$ is defined in $S$. Then we have
\[ (u,u\phi)(w,a)(v,v\phi) = (u,u\phi)(wv,a\lhd v + v\phi)
= (uwv,u\phi \lhd wv +(a \lhd v) +v\phi)
\]
and so applying $\pi_{\phi}$ gives 
\[ [(u,u\phi)(w,a)(v,v\phi)] \mapsto (uwv) \pi = w\pi=[(w,a)]\pi_{\phi} \,. \] 
So $\pi_{\phi}$ is well-defined, and $[w,a]_{\pi_{\phi}} \in E(Q) \Leftrightarrow w\in \mathbb{N}(Q)$. It is evident that
 $\pi_{\phi}$ is an ordered functor of ordered groupoids.

Now define the map $\A \xrightarrow{\iota} S\sslash M_{\phi}$  by $a \mapsto [(e,a)]$ for $a\in \A_e$.  To show injectivity of
 $\iota$, let $[(f,b)] \in S\sslash M_{\phi}$ with $b\in A_f$, and suppose that $[(e,a)]= [(f,b)]$.  Then
for some $u,v\in \mathbb{N}(Q)$ we have
\[ (u,u\phi)(e,a)(v,v\phi) = (uev,(u\phi ~ \lhd~ e~\lhd~ v)+ a~ \lhd ~ v+ v\phi) = (f,b) \,. \]
From the first component, we have $uev=f$ and so $e=f$, and $u=v^{-1}$. In the second component, $v$ acts
trivially on $\A$ and so
\[
(u\phi ~ \lhd~ e~\lhd~ v)+ a~ \lhd ~ v+ v\phi = u\phi + a + v\phi
= u\phi + a -  u\phi = a \,.\]
Hence $a=b$ and therefore $\iota$ is injective. 

We have that \begin{eqnarray*}
\ker(\pi_{\phi}) &= & \{ (w,a): w\in \mathbb{N}(Q) \} \\ &=&\{ [(w,w\phi)(w\textbf{r},(w\phi)^{-1}a) ] \} \\ &=& \{ [(w\textbf{r}
,(w\phi)^{-1}a)]\} \\ & \subseteq & \im (\iota)
\end{eqnarray*}  
Therefore $\mathscr{E}_{\phi}$ is an extension as desired.
\end{proof}

We denote $S\sslash M_{\phi}$ by $G_{\phi}$, and we denote the image in $G_{\phi}$ of an element $(w,a)\in S$ by 
$\left[ w,a\right]_{\phi}$. Note that $\left[w,a\right]_{\phi}\pi_{\phi}=w\pi$. For the trivial homomorphism $\tau \in
 \homogpd_{\mathbb{F}(Q)} (\mathbb{N}(Q), \A)$, we have $G_{\tau}= Q\ltimes \A$. 

Let $\mathscr{E}$ be the set of extensions constructed from the normal ordered subgroupoids $M_{\phi}$ of $S$ for $\phi \in
 \homogpd_{\mathbb{F}(Q)} (\mathbb{N}(Q), \A)$ occurring in Lemma \ref{exten}. Then the (abelian) group
 $\homogpd_{\mathbb{F}(Q)} (\mathbb{N}(Q), \A)$ acts transitively on $\mathscr{E}$: for $\alpha \in \homogpd_{\mathbb{F
}(Q)} (\mathbb{N}(Q), \A)$ we define $\mathscr{E}_{\phi} \lhd \alpha =\mathscr{E}_{\phi + \alpha}$. 

Since $\mathbb{N}(Q)$ acts trivially on $\A$, when restricted to $\mathbb{N}(Q)$ any derivation 
$\delta :\mathbb{F}(Q)\ra \A$ restricts to a morphism in $\homogpd_{\mathbb{F}(Q)} (\mathbb{N}(Q), \A)$. This gives us the
 mapping \[ \rho : \Der (\mathbb{F}(Q), \A) \ra \homogpd_{\mathbb{F}(Q)} (\mathbb{N}(Q), \A) \] in the five-term exact sequence
\[ 0\ra \Der_{}(Q,\A)\ra \Der_{\pi}(\mathbb{F}(Q),\A) \ra \Mod{Q}(\mathbb{N}^{ab}(Q),\A) \ra H^2(Q^I,\A^0) \xrightarrow{\varpi} H^2(\mathbb{F}(Q^I),\A^0) \] using the fact that 
$\Mod{Q}(\mathbb{N}^{ab}(Q),\A)= \homogpd_{\mathbb{F}(Q)} (\mathbb{N}(Q), \A)$. 

\begin{prop} 
The extensions $\mathscr{E}_{\phi}$ and $\mathscr{E}_{\psi}$ are equivalent if and only if $-\phi +\psi:\mathbb{N}(Q)\ra \A$ is the restriction of an ordered derivation $\mathbb{F}(Q)\ra \A$. 
\end{prop}

\begin{proof}Suppose that the extensions $\mathscr{E}_{\phi}$ and $\mathscr{E}_{\psi}$ are equivalent: that is, there exists an ordered functor $\mu :G_{\phi} \ra G_{\psi}$ making the diagram 
\begin{center} 
\begin{tikzcd} & G_{\phi}\ar{dd}{\mu} \ar{dr}{\pi_{\phi}} & \\ \A \ar{ru} \ar{rd}& & Q \\ &G_{\psi} \ar{ru}{\pi_{\psi}}& \end{tikzcd}
\end{center} 
commute. Then for all $a\in A_e$ and $(w,b)\in S$ we have:
\begin{align}
\left[ e,a\right]_{\phi}\mu &= \left[ e,a\right]_{\psi}, \label{calll} \\
\left[ w,b\right]_{\phi}\mu \pi_{\psi} = \left[ w,b\right]_{\phi}\pi_{\phi} &= w\pi = \left[ w,b\right]_{\psi}\pi_{\psi}.
\end{align}Therefore $\left[w,b\right]_{\psi}^{-1}\left[w,b\right]_{\phi}\mu \in \ker \pi_{\psi}$. Setting $b=w\textbf{r}$, we deduce that 
\begin{eqnarray}\label{cal}
\left[w,wr\right]_{\psi}^{-1}\left[ w,w\textbf{r}\right]_{\phi}\mu =\left[w\textbf{r}, a\right]_{\psi}
\end{eqnarray} for some $a\in A_{w\textbf{r}}$. Since the map $\A\ra G_{\psi}$, $a\mapsto \left[a\textbf{r},a\right]_{\psi}$ is injective, the element $a$ determined by equation \eqref{cal} is unique. Hence we get a mapping $\delta:\mathbb{F}(Q) \ra \A$ given by $w\mapsto a$, with $a\in \A_{w\textbf{r}}$, and satisfying
\begin{eqnarray}\label{yea}
 \left[w\textbf{r},w\delta \right]_{\psi}= \left[w,w\textbf{r}\right]_{\psi}^{-1} \left[w,w\textbf{r}\right]_{\phi} \mu 
\end{eqnarray}
We show that $\delta$ is an ordered derivation. Given $u,v\in \mathbb{F}(Q)$ such that the composition $uv$ exists, we have 
\begin{eqnarray*}
\left[ v\textbf{r},(uv)\delta\right]_{\psi}&=&\left[uv,v\textbf{r}\right]_{\psi}^{-1}\left[ uv,v\textbf{r}\right]_{\phi}\mu \\
&=& \left[v,v\textbf{r}\right]_{\psi}^{-1} \left[u,u\textbf{r}\right]_{\psi}^{-1} \left[ u,u\textbf{r}\right]_{\phi}\mu \left[v,v\textbf{r}\right]_{\phi} \mu \\
&=& \left[ v,v\textbf{r}\right]_{\psi}^{-1}\left[u,u\textbf{r}\right]_{\psi}^{-1}\left[u,u\textbf{r}\right]_{\phi}\mu \left[v,v\textbf{r}\right]_{\psi} \left[v,v\textbf{r}\right]_{\psi}^{-1}\left[v,v\textbf{r}\right]_{\phi}\mu \\
&=& \left[v,v\textbf{r}\right]_{\psi}^{-1}\left[u\textbf{r},u\delta \right]_{\psi} \left[v,v\textbf{r}\right]_{\psi} \left[v\textbf{r},v\delta \right]_{\psi} \quad \text{by \eqref{yea}}\\
&=& \left[v\textbf{r},(u\delta)\lhd v+v\delta \right]
\end{eqnarray*}  and therefore $(uv)\delta=(u\delta)\lhd v+v\delta$. Hence $\delta$ is an ordered derivation $\mathbb{F}(Q)\ra \A$. We now show that $(-\phi + \psi)$  is a restriction of $\delta$.

Suppose that $w\in \mathbb{N}(Q)$. Then recalling that $w\textbf{d}=w\textbf{r}$ we have \begin{eqnarray}
\left[ w,w\textbf{r}\right]_{\phi}= \left[ (w,w\textbf{r})(w^{-1},w^{-1}\phi)\right]_{\phi}=\left[w\textbf{r}, w^{-1}\phi \right]_{\phi}
\end{eqnarray} and so \begin{eqnarray*}
\left[w\textbf{r},w\delta \right]_{\psi}&=& \left[ w\textbf{d},w^{-1}\psi \right]_{\psi}^{-1} \left[w\textbf{r},w^{-1}\phi \right]_{\phi} \mu \\
&=& \left[ w\textbf{r},w\psi \right]_{\psi} \left[w\textbf{r},w^{-1}\phi \right]_{\phi} \mu \\
&=& \left[ w\textbf{r}, w\psi \right]_{\psi} \left[ w\textbf{r},w^{-1}\phi \right]_{\psi} \,. \end{eqnarray*} 
Using \eqref{calll}, we have
$\left[w\textbf{r},w\delta \right]_{\psi}= \left[ w\textbf{r},w\psi -w\phi \right]$ and, since 
$a\mapsto \left[ a\textbf{r},a\right]_{\psi}$ is injective, we deduce that $w\delta =w(-\phi +\psi)$. Therefore if the extensions $\mathscr{E}_{\phi}$ and $\mathscr{E}_{\psi}$ are equivalent then $(-\phi +\psi)$ is the restriction of an ordered derivation $\delta :\mathbb{F}(Q)\ra \A$. 

For the converse, let $\eta:\mathbb{F}(Q) \ra \A$ be an ordered derivation and $\phi \in \homogpd_{\mathbb{F}(Q)} (\mathbb{N}(Q), \A)$. Since $\eta$ restricts to an $\mathbb{F}(Q)$--map $\mathbb{N}(Q)\ra \A$ and $\homogpd_{\mathbb{F}(Q)} (\mathbb{N}(Q), \A)$ is closed under pointwise addition, the sum  $\phi +\eta \in \homogpd_{\mathbb{F}(Q)} (\mathbb{N}(Q), \A)$. Set $\psi =\phi +\eta$ and define a mapping $\nu :G_{\phi}\ra G_{\psi}$ by $\left[w,b\right]_{\phi} \mapsto \left[w,w\eta +b\right]_{\psi}$. Then for $u,v\in \mathbb{N}(Q)$ we have \[(u,u\phi)(w,b)(v,v\phi)=(uw,(u\phi)\lhd w+b)(v,v\phi)=(uwv,(u\phi)\lhd w+b+v\phi) \] since $\mathbb{N}(Q)$ acts trivially on $\A$, and so \begin{eqnarray*}
\left[(u,u\phi)(w,b)(v,v\phi)\right]_{\phi}\nu &=& \left[ uwv,(u\phi) \lhd w+b+v\phi \right]_{\phi}\nu  \\
&=& \left[ uwv,(uwv)\eta + (u\phi)\lhd w+b+v\phi \right]_{\psi} \\
&=& \left[ uwv, (u\eta) \lhd wv +(wv)\eta +(u\phi)\lhd w+b+v\phi \right]_{\psi} \\
&=& \left[ uwv, (u\eta)\lhd w+w\eta +v\eta + (u\phi)\lhd w+b+v\phi \right]_{\psi} \\
&=& \left[ uwv,(u\psi)\lhd w+v\psi +w\eta +b \right]_{\psi} \\
&=&\left[ (u,u\psi)(w,w\eta +b)(v,v\psi)\right]_{\psi} \\
&=& \left[ w,w\eta +n \right]_{\psi} = \left[ w,b\right]_{\phi}\nu 
\end{eqnarray*} and so $\nu$ is well-defined. It is then easy to see that $\nu$ is an ordered functor, and that 
\begin{center}
\begin{tikzcd} & G_{\phi}\ar{dd}{\nu} \ar{dr}{\pi_{\phi}} & \\ \A \ar{ru} \ar{rd}& & Q \\ &G_{\psi} \ar{ru}{\pi_{\psi}}& \end{tikzcd}
\end{center}
commutes, so that $\mathscr{E}_{\phi}$ and $\mathscr{E}_{\psi}$ are equivalent.
\end{proof}

\begin{cor}\label{one} 
The set of equivalence classes of extensions of $\A$ by $Q$ in $\mathscr{E}$ is in one-to-one correspondence with the cokernel of the restriction map 
\[ \rho: \Der(\mathbb{F}(Q),\A)\ra \homogpd_{\mathbb{F}(Q)} (\mathbb{N}(Q), \A) \; .\] 
\end{cor}

\begin{prop}\label{ex} An extension $\A\xrightarrow{\iota} G\xrightarrow{\varphi}Q$ is equivalent to an extension $\mathscr{E}_{\alpha}$ for some $\alpha \in \homogpd_{\mathbb{F}(Q)} (\mathbb{N}(Q), \A)$. 
\end{prop}

\begin{proof}
Lift the quotient map $\pi$ to $\pi_{*}:\mathbb{F}(Q)\ra G$ making the diagram 
\begin{center} 
\begin{tikzcd} 
\mathbb{N}(Q)\rar &\mathbb{F}(Q) \ar{dr}{\pi} \ar{dd}{\pi_*} & \\ &&Q \\ \A\ar{r}{\iota} &G \ar{ru}{\varphi} & 
\end{tikzcd} 
\end{center} 
commute. Then $\mathbb{N}(Q)\pi_* \subseteq \A$ and we may define $\xi: \mathbb{F}(Q)\ltimes \A \ra G$ by $(w,a)\xi= (w\pi_*)(a\iota)$ and $\ker \xi = \{ (w,w^{-1}\pi_*): w\in \mathbb{N}(Q) \}$.
Since each $A_e$ is abelian, the map $\alpha :w\mapsto w^{-1}\pi_*$ is a homomorphism $\mathbb{N}(Q)\ra \A$, and $\mathscr{E}$ is equivalent to $\mathscr{E}_{\alpha}$.
\end{proof}

\begin{remark}
In \cite{Mthesis}, Matthews discusses the concept of factor sets for modules over ordered groupoids. We note that we can recover the factor set discussed by Matthew from our construction. The restriction of the map $\pi_*:\mathbb{F}(Q)\ra G$ to the elements $\lfloor q \rfloor^{-1} \cdot \lfloor (p|qq^{-1})\rfloor^{-1} \cdot  \lfloor p\ast q\rfloor \in \mathbb{N}(Q)$ is a factor set discussed in \cite{Mthesis}.
\end{remark}

\subsection{Cohomology and extensions of ordered groupoids}
\begin{theorem}\label{onetoone} 
The set of equivalence classes of extensions of $\A$ by $Q$ is in one-to-one correspondence with the second cohomology group $H^2(Q^I,\A^0)$. 
\end{theorem}

\begin{proof} 
From the extension  $\mathbb{N}(Q)\ra \mathbb{F}(Q)\xrightarrow{\pi} Q$ of ordered groupoids we obtain, by
Lemma \ref{KQ}, the short exact sequence  $\mathbb{N}^{ab}(Q) \ra D_{\pi} \ra KQ$ of $Q$--modules. By Theorem \ref{five} and using the fact that $\Mod{Q}(\mathbb{N}^{ab}(Q),\A)= \homogpd_{\mathbb{F}(Q)} (\mathbb{N}(Q), \A)$ we obtain the five--term exact sequence 
\begin{align}
\label{free_five}
 0\ra \Der_{}(Q,\A)\ra \Der_{\pi}(\mathbb{F}(Q),\A) \ra \homogpd_{\mathbb{F}(Q)} & (\mathbb{N}(Q), \A) \ra \nonumber\\
&~ H^2(Q^I,\A^0) \xrightarrow{\varpi} H^2(\mathbb{F}(Q^I),\A^0)\; . \end{align} 
For any $Q^I$--module $\B$, a $Q^I$--map $K\mathbb{F}(Q^I) \ra \B$ corresponds to a homomorphism $\mathbb{F}(Q^I)\ra \mathbb{F}(Q^I) \ltimes \B$. If $\A \ra \B$ is an epimorphism of $Q^I$--modules then we obtain a lift 
\begin{center}
\begin{tikzcd} & \mathbb{F}(Q^I)\ar{d} \ar{dl}  \\ \mathbb{F}(Q^I) \ltimes \A \ar{r}& \mathbb{F}(Q^I) \ltimes \B \end{tikzcd}
\end{center} 
via the freeness of $\mathbb{F}(Q^I)$, and so we get the corresponding lift $K\mathbb{F}(Q^I) \ra \A$. Thus $K\mathbb{F}(Q^I)$ is projective and hence the sequence $K\mathbb{F}(Q^I) \ra \Z \mathbb{F}(Q^I) \ra \Delta \mathbb{Z}$ is a projective resolution of $\Delta \mathbb{Z}$. It follows that $H^n(\mathbf{F}(Q^I),\A^0)=0$ for $n>0$, and the sequence \eqref{free_five}
becomes 
\[ 0\ra \mathrm{Der}(Q,\A)\ra \mathrm{Der}_{\pi}(\mathbb{F}(Q),\A) \ra \homogpd_{\mathbb{F}(Q)} (\mathbb{N}(Q), \A) \xrightarrow{\varpi} H^2(Q^I,\A^0) \ra 0 \; . \] 
Therefore $\varpi$ is a bijection from the cokernel of the restriction map $\mathrm{Der}_{\pi}(\mathbb{F}(Q),\A) \ra \homogpd_{\mathbb{F}(Q)} (\mathbb{N}(Q), \A)$ to $H^2(Q^I,\A^0)$  and thus by corollary \ref{one} the result follows.
\end{proof}

\end{document}